\documentclass{amsproc}
\usepackage{xy, amssymb, hyperref}
\hypersetup{pdfauthor={Amritanshu Prasad}, pdftitle={Locally compact abelian groups with symplectic self-duality}}
\xyoption{all}

\numberwithin{equation}{section}
\theoremstyle{plain}
\newtheorem{theorem}[equation]{Theorem}
\newtheorem*{theorem*}{Theorem}
\newtheorem{prop}[equation]{Proposition}
\newtheorem*{prop*}{Proposition}
\newtheorem{lemma}[equation]{Lemma}
\newtheorem*{lemma*}{Lemma}
\newtheorem{cor}[equation]{Corollary}
\newtheorem*{cor*}{Corollary}

\newtheorem*{fq*}{Main question}

\theoremstyle{definition}

\newtheorem*{defn*}{Definition}
\newtheorem{remark}[equation]{Remark}
\newtheorem*{remark*}{Remark}
\newtheorem{example}[equation]{Example}

\newcommand{\ctorsion}[1]{\Z/p^{#1}}
\newcommand{\dash}{\nobreakdash-}
\newcommand{\dtorsion}[1]{\T[p^{#1}]}

\newcommand{\Heis}{\mathbf H}

\newcommand{\id}[1]{\mathrm{id}_{#1}}
\newcommand{\inv}{^{-1}}
\newcommand{\lca}{{locally compact abelian group}}
\newcommand{\mat}[4]{
    \begin{pmatrix}
      #1 & #2 \\
      #3 & #4
    \end{pmatrix}
}

\newcommand{\Q}{\mathbf Q}
\newcommand{\R}{\mathbf R}
\newcommand{\ses}[5]{
  \xymatrix{
    0 \ar[r] & #1 \ar[r]^{#4} & #2 \ar[r]^{#5} & #3 \ar[r] & 0
  }
}
\newcommand{\xses}[6]{
  \tag{$#1$}
  \ses{#2}{#3}{#4}{#5}{#6}
}

\newcommand{\T}{\mathbf T}
\newcommand{\toptor}{\mathrm{toptor}}
\newcommand{\transpose}[1]{#1^{\textsc{t}}}
\newcommand{\Z}{\mathbf Z}

\DeclareMathOperator{\Ext}{Ext}
\DeclareMathOperator{\rank}{rank}

\DeclareMathOperator{\Hom}{Hom}
\DeclareMathOperator*{\rdp}{{\prod}^\prime}

\title[Symplectic self-duality]{Locally Compact Abelian Groups\\with symplectic self-duality}
\author{Amritanshu Prasad}
\address{The Institute of Mathematical Sciences, Chennai.}
\email{amri@imsc.res.in}
\author{Ilya Shapiro}
\address{Institut des Hautes \'Etudes Scientifiques, Bures-sur-Yvette.}
\email{shapiro@ihes.fr}
\author{M.~K.~Vemuri}
\address{Chennai Mathematical Institute, Siruseri.}
\email{mkvemuri@gmail.com}
\begin{document}
\begin{abstract}
  Is every locally compact abelian group which admits a symplectic self-duality isomorphic to the product of a locally compact abelian group and its Pontryagin dual?
  Several sufficient conditions, covering all the typical applications are found.
  Counterexamples are produced by studying a seemingly unrelated question about the structure of maximal isotropic subgroups of finite abelian groups with symplectic self-duality (where the original question always has an affirmative answer).
\end{abstract}
\maketitle
\section{The main questions}
\label{sec:main_question}
Let $\T$ denote the circle group, thought of as $\R/\Z$.
The Pontryagin dual of a \lca{} is the \lca{} $\hat L$ of continuous homomorphisms $L\to \T$, endowed with the compact open topology.
\begin{defn*}
  [Symplectic self-duality]
  A self-duality of a \lca{} $L$ is an isomorphism $\nabla:L\to \hat L$. If $\nabla(x)(x)=0$ for all $x\in L$, then $\nabla$ is called a symplectic self-duality.
\end{defn*}
\begin{defn*}
  [Isotropic subgroup]
  Let $\nabla$ be a self-duality of a \lca{} $L$.
  A subgroup $N$ of $L$ is said to be isotropic if $\nabla(x)(y)=0$ for all $x,y\in N$.
\end{defn*}
\begin{defn*}
  [Isomorphism of self-dualities]
  Let $\nabla$ and $\nabla'$ be self-dualities of \lca{s} $L$ and $L'$ respectively.  The pairs
  $(L,\nabla)$ and $(L',\nabla')$ are said to be isomorphic if there exists an isomorphism $\phi:L\to L'$ such that
  \begin{equation*}
    \nabla'(\phi(x))(\phi(y))=\nabla(x)(y).
  \end{equation*}
  If $M$ and $M'$ are subgroups in $L$ and $L'$ respectively, and $\phi$ above can be chosen so that $\phi(M)=M'$, then $(L,\nabla,M)$ and $(L',\nabla',M')$ are said to be isomorphic triples.
\end{defn*}
\begin{example}
  \label{eg:main-questions}
  For any \lca{} $A$,
  let $L^0=A\times \hat A$, and
  define $\nabla^0:L^0\to \hat L^0$ by
  \begin{equation*}
    \nabla^0(x,\chi)(y,\lambda)=\chi(y) - \lambda(x) \text{ for } x,y\in A \text{ and } \chi,\lambda \in \hat A.
  \end{equation*}
  Then, $\nabla^0$ is a symplectic self-duality, which will be called the standard symplectic self-duality of $L^0$.
  Let $B$ be any closed subgroup of $A$, in particular it can be $A$ itself or $0$.
  Define
  \begin{equation*}
    B^\perp=\{\chi\in \hat A\;|\;\chi(x)=0 \text{ for all } x\in B\}.
  \end{equation*}
  Then $M^0=B\times B^\perp$ is a maximal isotropic subgroup of $L^0$.
\end{example}
\begin{defn*}
  [Standard pairs and triples]
  Let $\nabla$ be a symplectic self-duality of a \lca{} $L$.
  We say that $(L,\nabla)$ is a standard pair (or that $\nabla$ is standard) if it is isomorphic to some $(L^0,\nabla^0)$ as in Example~\ref{eg:main-questions}.
  Let $M$ be a maximal isotropic subgroup of $L$.  Then
  $(L,\nabla,M)$ is said to be a standard triple if it is isomorphic to some $(L^0,\nabla^0,M^0)$ as in Example~\ref{eg:main-questions}.
\end{defn*}
Locally compact abelian groups with symplectic self-duality arise naturally in the context of Heisenberg groups (see Section~\ref{sec:Heis}).
The applications of Heisenberg groups to number theory and representation theory usually involve locally compact abelian groups which are based on finite fields, local fields, or the ad\`eles of a global field.
In all these cases, the \lca{s} with symplectic self-duality are standard (see Theorem~\ref{theorem:exponent-p}, Corollaries~\ref{cor:fdvs}, \ref{cor:finite-p-rank} and \ref{cor:finite}, and Theorem~\ref{theorem:divisible}).
The question of whether symplectic self-dualities of \lca s are always standard is of independent interest in view of the fact that despite some progress (see Section~\ref{sec:survey}) detailed structure theorems for general self-dual locally compact abelian groups remain elusive.
This question, which has been open until now, is resolved here.
In fact, we exhibit a group which admits a symplectic self-duality but is not isomorphic to $A\times \hat A$ for any \lca{} $A$ (Theorem~\ref{theorem:non-standard-pair}).

\section{Overview}
\label{sec:survey}
Although there is no classification of self-dual \lca{s} in general, a structure theorem for the metrizable torsion free ones was obtained by M.~Rajagopalan and T.~Soundararajan \cite{MR0247375}.
Wu \cite{MR1173767} was able to weaken the hypotheses under which the structure theorem of \cite{MR0247375} is valid, and provided a structure theorem for divisible self-dual locally compact abelian groups.

The classical Pontryagin-van Kampen structure theorem \cite[Theorem~2]{MR1503234} states that any \lca{} can be written as a product of the additive group of a finite dimensional real vector space and a \lca{} which contains a compact open subgroup.
This theorem can be used to reduce the study of self-dual \lca{s} to those which contain a compact open subgroup (see \cite[Proposition~6.5]{MR1651168}).
Such groups may be regarded as extensions over discrete abelian groups with kernels that are compact abelian groups.
In the context of self-duality extensions of the form
\begin{equation*}
  \ses GL{\hat G}fq
\end{equation*}
are of particular interest.
In particular, one may look at those extensions for which the dual extension
\begin{equation*}
  \ses GL{\hat G}{\hat q}{\hat f}
\end{equation*}
is equivalent to the original one (known as self-dual or autodual extensions).
Such self-dual groups were studied by Lawrence Corwin \cite{MR0269775} using $H^2(\hat G,G)$ and by L\'aszl\'o Fuchs and Karl H.~Hofmann \cite{MR1475126,MR1651168} using $\Ext^1(\hat G,G)$.
Corwin characterized the cocycles which give rise to autodual extensions.
He then used this characterization to give a construction for autodual extensions.
Fuchs and Hofmann \cite{MR1475126} characterized the compact abelian groups $G$ for which all extensions over $\hat G$ with kernel $G$ are self-dual.
They showed \cite{MR1651168} that groups which have a symmetric self-duality (an isomorphism $\nabla:L\to \hat L$ such that $\hat \nabla=\nabla$) and a compact open isotropic subgroup $G$ which is its own orthogonal complement are precisely those which are autodual extensions over $\hat G$ with kernel $G$ (under the hypothesis that $L$ is $2$\dash regular, i.e., multiplication by $2$ is an automorphism of $L$).
Whether or not every self-dual \lca{} admits a strong self-duality is an open question.

R.~Ranga Rao addressed the question of when \lca{s} with symplectic self-duality are standard in \cite{MR1280069}, where he tried to identify a class of \lca{s} where all symplectic self-dualities are standard.
While his paper contains valuable insights, the main theorem is false (Theorem~\ref{theorem:non-standard-pair} gives a counterexample).

The first three sections of this paper are introductory in nature, with Section~\ref{sec:Heis} explaining the relation between \lca{}s with symplectic self-duality and the theory of Heisenberg groups.

Section~\ref{sec:general-lemmas} collects some general lemmas which are used later.
The standardness of all symplectic self-dualities of groups of prime exponent is also established here.
In Section~\ref{sec:peeling}, the main tool is a lemma of Ranga Rao (Lemma~\ref{peeeling-lemma}).
This lemma is systematically applied (Corollaries~\ref{cor:fdvs}--\ref{cor:finite}) to easily prove the standardness of symplectic self-dualities for finite dimensional vector spaces over certain self-dual topological fields and groups of finite $p$\dash rank (including finite groups).

In Section~\ref{sec:red-open-compact} the study of \lca{s} with symplectic self-duality is reduced to groups which have a compact open subgroup.
In fact, standardness questions are reduced to the residual part (see Remark~\ref{remark:residual}).
As suggested by the work of Fuchs and Hofmann \cite{MR1651168}, symplectic self-dualities of such groups are studied using antidual extensions.
Theorem~\ref{theorem:anti-dual-extensions}, which is the symplectic analogue of \cite[Theorem~5.7A]{MR1651168}, characterizes \lca{s} with compact open subgroups and symplectic self-dualities in terms of antidual extensions.
It turns out that although the work of Corwin and of Fuchs and Hofmann cited above concentrates largely on symmetric self-duality, the method that is used in this work is actually even better suited to the study of symplectic self-duality.
The reason is that, unlike the situation for symmetric self-dualities,  every \lca{} with a symplectic self-duality and a compact open subgroup has a compact open isotropic subgroup which is its own orthogonal complement (Lemma~\ref{lemmma:max-iso}).
Section~\ref{sec:extensions-products} discusses the behaviour of symplectic self-duality under product decompositions of a compact open maximal isotropic subgroup.
With the help of this analysis, a simple homological criterion for standardness of triples is obtained in Section~\ref{sec:standard}.

The homological criterion for standardness (Theorem~\ref{theorem:standard-triples}) tempts one into thinking that it holds the key to proving the standardness of all symplectic self-dualities.
It turns out that non-standard triples exist even in the finite case (where all symplectic self-dualities are known to be standard).
Such examples are constructed in Section~\ref{sec:finite} by reducing the problem to one of matrix reduction.
The matrix reduction techniques can also be used to establish the standardness of triples when the compact open maximal isotropic subgroup is $(\Z/p^k\Z)^l$ or $\Z_p^l$ (Theorems~\ref{theorem:homogeneous} and~\ref{theorem:Z_p^l}).

In Section~\ref{sec:topol-tors-groups}, a criterion for the standardness of topological torsion groups (Theorem~\ref{theorem:toptor}) in terms of their primary decomposition is obtained and used (together with the examples of non-standard triples in Section~\ref{sec:finite}) to construct examples of \lca s with symplectic self-dualities which are not isomorphic to $A\times \hat A$ for any \lca{} $A$ (Theorem~\ref{theorem:non-standard-pair}).
Theorem~\ref{theorem:toptor} is also used to prove the standardness of symplectic self-dualities of divisible groups (Theorem~\ref{theorem:divisible}).
The standardness of a symplectic self-duality of a \lca{} is often reduced to its topological torsion part (Theorem~\ref{theorem:reduction-to-toptor} and Corollary~\ref{cor:torus-toptor}).
Finally it is shown that all symplectic self-dualities of \lca s with finite topological torsion part are standard (Corollary~\ref{cor:finite-toptor}).

It is worth pointing out that, in all of our standardness results (Theorem~\ref{theorem:exponent-p}, Proposition~\ref{prop:Krull} and its corollaries, Theorem~\ref{theorem:divisible} and Corollary~\ref{cor:finite-toptor}), it can be inferred that all of the symplectic self-dualities on the locally compact abelian group under consideration are isomorphic.

\noindent \textbf{Acknowledgements.} We are deeply indebted to Karl H.~Hofmann whose insightful report on an early attempt by A.P. and M.K.V. pointed us in the right direction.
His suggestions have influenced the writing to this article to a large extent.
A.P. and I.S. thank the IH\'ES for its hospitality, and for the opportunity to collaborate.
I.S. also thanks the University of Waterloo.
\section{Motivation: Heisenberg groups}
\label{sec:Heis}
Heisenberg groups were first introduced by Herman Weyl in his mathematical formulation of quantum kinematics \cite[Chap.~IV, Section~14]{Weyl50} as certain central extensions of locally compact abelian groups by $\T$.
Weyl envisioned that the abelian groups concerned could be quite general, and besides real vector spaces, he considered finite abelian groups.
The idea that Heisenberg groups provide an elegant conceptual framework within which the foundations of harmonic analysis on locally compact abelian groups can be developed has its genesis in Mackey's work \cite{Mackey49}, and is now well established.
A discussion of this and of some other areas of mathematics where Heisenberg groups play a prominent role can be found in \cite{MR578375}.

Mackey's conception of a Heisenberg group as described by Weil in \cite[Section~4]{MR0165033} is the following: given a locally compact abelian group $A$, for each element $(x,\chi)\in A\times \hat A$, consider the unitary Weyl operator $U(x,\chi)$ on $L^2(A)$ given by
\begin{equation*}
  U(x,\chi)f(z)=e^{2\pi i\chi(z)}f(z-x).
\end{equation*}
The composition of such operators is given by
\begin{equation*}
  U(x,\chi)U(y,\lambda)=e^{-2\pi i \lambda(x)}U(x+y,\chi+\lambda).
\end{equation*}
Therefore, the operators $e^{2\pi i t}U(x,\chi)$, for $t\in \T$, $x\in A$ and $\chi\in \hat A$ form a subgroup of the group $U(L^2(A))$ of unitary operators on $L^2(A)$ (endowed with the strong operator topology).
When $\T\times A\times \hat A$ is given the product topology, $(t,x,\chi)\mapsto e^{2\pi i t}U(x,\chi)$ defines a homeomorphic map from $\T\times A\times \hat A$ onto its image in $U(L^2(A))$.
The resulting group $\Heis$ is known as a Heisenberg group and is a central extension of the form
\begin{equation*}
  \ses{\T}\Heis{A\times \hat A}{}{}.
\end{equation*}
Note that the commutator on $\Heis$ is given by
\begin{equation*}
U(x,\chi)U(y,\lambda)U(x,\chi)\inv U(y,\lambda)\inv = e^{2\pi i(\chi(y)-\lambda(x))}U(0,0).
\end{equation*}
It is for these groups that Mackey proved the Stone-von Neumann theorem: that there is (up to unitary equivalence) only one irreducible unitary representation of $\Heis$ with central character $t\mapsto e^{2\pi i t}$ (called the canonical representation).

It is natural to ask, for which locally compact abelian groups does there exist a central extension by $\T$ such that there is (up to unitary equivalence) a unique irreducible unitary representation where $t\in \T$ acts as scalar multiplication by $e^{2\pi it}$.
Let $\Heis$ be a central extension:
\begin{equation}
  \label{eq:9}
  \ses{\T}\Heis L{}{}
\end{equation}
which admits a continuous section over $L$.
Thus
$\Heis=\T\times L$ as a set, and the group law is of the form
\begin{equation*}
  (t,x)(u,y)=(t+u+\psi(x,y),x+y)
\end{equation*}
for some $\psi:L\times L\to \T$ satisfying the cocycle condition
\begin{equation*}
  \psi(x,y)+\psi(x+y,z)=\psi(x,y+z)+\psi(y,z)\text{ for all } x,y,z\in L.
\end{equation*}
The commutator map $(g,h)\mapsto ghg\inv h\inv$, which is a function $\Heis\times \Heis\to \T$, descends to an alternating bilinear map $B:L\times L\to \T$ given by
\begin{equation*}
  B(x,y)=\psi(x,y)-\psi(y,x)
\end{equation*}
which is known as the commutator form for the extension (\ref{eq:9}).
Define $\nabla_B:L\to \hat L$ by $\nabla_B(x)(y)=B(x,y)$.
Then $\Heis$ has a unique (up to unitary equivalence) irreducible unitary representation such that $t\in T$ acts as scalar multiplication by $e^{2\pi it}$ if and only if $\nabla_B$ is an isomorphism.
In \cite[Defn.~1.1]{MR1116553}, this condition is taken as the definition of a Heisenberg group.

Suppose that $\Heis'$ is another Heisenberg group of the form (\ref{eq:9}) with the product law
\begin{equation*}
  (t,x)(u,y)=(t+u+\psi'(x,y), x+y)
\end{equation*}
for some other $2$-cocycle $\psi':L\times L\to\T$ for which
\begin{equation*}
  \psi'(x,y)-\psi'(y,x)=B(x,y)
\end{equation*}
is the same alternating bilinear map as for $\Heis$.
Consider the central extension $\Heis_0$
\begin{equation}
  \label{eq:10}
  \ses{\T}{\Heis_0}L{}{}
\end{equation}
where the group structure on the set $\Heis_0=\T\times L$ is given by
\begin{equation}
  \label{eq:11}
  (t,x)(u,y)=(t+u+\psi(x,y)-\psi'(x,y),x+y).
\end{equation}
The commutator form for $\Heis_0$ is given by
\begin{equation*}
  B_0(x,y)=(\psi(x,y)-\psi'(x,y))-(\psi(y,x)-\psi'(y,x))=B(x,y)-B(x,y)=0.
\end{equation*}
Therefore $\Heis_0$ is a locally compact abelian group.
Since $\T$ splits from any \lca{} \cite[6.16]{0509.22003}, so (\ref{eq:10}) splits.
Suppose that the splitting section $s:\T\to \Heis_0$ is given by $x\mapsto (\alpha(x),x)$.
Expanding the identity $s(x+y)=s(x)s(y)$ using (\ref{eq:11}) gives
\begin{equation}
  \label{eq:12}
  \alpha(x+y)-\alpha(x)-\alpha(y)=\psi(x,y)-\psi'(x,y).
\end{equation}
Define $\phi:\Heis\to \Heis'$ by $(t,x)\mapsto (t-\alpha(x),x)$.
Then,
\begin{equation*}
  \phi((t,x)(u,y))=\phi(t+u+\psi(x,y),x+y)=(t+u+\psi(x,y)-\alpha(x+y),x+y)
\end{equation*}
On the other hand,
\begin{equation*}
  \phi(t,x)\phi(u,y)=(t-\alpha(x),x)(u-\alpha(y),y)=(t+u-\alpha(x)-\alpha(y)+\psi'(x,y),x+y).
\end{equation*}
These are equal by (\ref{eq:12}).
Therefore, we have an isomorphism of central extensions
\begin{equation*}
  \xymatrix{
    0 \ar[r] & \T \ar[r] \ar@{=}[d] & \Heis  \ar[r] \ar[d]^\phi & L \ar[r] \ar@{=}[d] & 0\\
    0 \ar[r] & \T \ar[r] & \Heis' \ar[r] & L \ar[r] & 0.
  }
\end{equation*}
Conversely, it is easy to see that isomorphic central extensions of the type (\ref{eq:9}) give rise to the same alternating bicharacter.
A well-known result is obtained:
\begin{prop}
  \label{prop:motiv-heis-groups}
  The commutator form gives rise to an injective map from the set of equivalence classes of Heisenberg groups whose quotient modulo $\T$ is $L$ and the set of symplectic self-dualities $L\to \hat L$.
\end{prop}
Under the above correspondence, Mackey's Heisenberg groups correspond to standard pairs.

Realizations of the canonical representation of a Heisenberg group where $t\in \T$ acts as multiplication by $\theta(t)=e^{2\pi it}$ are obtained by extending $\theta$ to a unitary character of a maximal abelian subgroup of $\Heis$ and then inducing to $\Heis$.
The maximal abelian subgroups of $\Heis$ are precisely the pre-images of the maximal isotropic subgroups of $L$ with respect to the commutator form.
The perfect realizations of the canonical representation of $\Heis$ correspond precisely to systems of imprimitivity, as is explained in \cite{pre05272580}.
The maximal systems of imprimitivity for the canonical representation are expected to correspond to the maximal isotropic subgroups of $L$.
This has been shown when $L$ is a finite dimensional  real vector space in \cite{pre05272580} and when $L$ is a finite abelian group in \cite{iafhg}.
\section{Generalities}
\label{sec:general-lemmas}
Suppose that $L$ has a product decomposition $L=\prod_i L_i$.
Then $\hat L=\bigoplus \widehat{L_i}$, where $\widehat{L_i}$ is identified with the characters of $L$ which vanish on the kernel of the projection $L\to L_i$.

If $L$ is a finite product $\prod_{i=1}^n L_i$, and $\nabla_i$ is a self-duality of $L_i$, then $L$ inherits a symplectic self-duality $\nabla$ by setting
\begin{equation*}
  \nabla(x_1,\ldots,x_n)(y_1,\ldots,y_n)=\sum_{i=1}^n \nabla_i(x_i)(y_i).
\end{equation*}
The pair $(L,\nabla)$ is called the orthogonal sum of the $(L_i,\nabla_i)$'s.
Note that the orthogonal sum of standard pairs is standard.

If $\nabla$ is a self-duality of $L$, then $A$ is an isotropic subgroup of $L$ if and only if $A\subset \nabla\inv(A^\perp)$ (recall that $A^\perp$ is the subgroup of $\hat L$ consisting of characters which vanish on $A$).
If $A=\nabla\inv(A^\perp)$ then $A$ is a maximal isotropic subgroup.

The following result is well-known, and very easy to prove:
\begin{lemma}
  \label{lemma:subquotient}
  Let $\nabla$ be a self-duality of a \lca{} $L$.
  Suppose that $A$ is a closed isotropic subgroup of $L$.
  Let $L_A=\nabla\inv(A^\perp)/A$.
  Then $A^\perp/\nabla(A)$ is canonically isomorphic to the Pontryagin dual of $L_A$, and $\nabla$ descends to a well-defined isomorphism $\nabla_A:L_A\to \widehat{L_A}$.
  If $\nabla$ is a symplectic or symmetric self-duality, then so is $\nabla_A$.
\end{lemma}
If $L=\bigoplus_{j\in J} L_j$ and $L'=\prod_{i\in I} L'_i$ (possibly infinite decompositions) every homomorphism $\phi:L\to L'$ can be expressed uniquely in the form
\begin{equation*}
  \phi\Big(\sum_j x_j\Big)_i=\sum_j \phi_{ij}(x_j),
\end{equation*}
where $\phi_{ij}:L_j\to L'_i$ is a homomorphism for each $i\in I$ and $j\in J$.
When $|I|$ and $|J|$ are finite it is convenient to denote $\phi$ by the $|I|\times |J|$ matrix $(\phi_{ij})$.
\begin{lemma}
  \label{lemma:polarization-implies-standard}
  Let $\nabla$ be a symplectic self-duality of a \lca{} $L$.
  If there exist isotropic subgroups $L_1$ and $L_2$ of $L$ such that $L=L_1\times L_2$ then $(L,\nabla,L_1)$ is a standard triple.
\end{lemma}
\begin{proof}
  As explained at the beginning of this section, $\hat L$ inherits a decomposition $\widehat{L_1}\times \widehat{L_2}$.
  Using matrix notation for $\nabla$ with respect to these decompositions of $L$ and $\hat L$, the hypothesis implies that
  \begin{equation*}
    \nabla = \mat 0{\nabla_{12}}{-\widehat{\nabla_{12}}}0,
  \end{equation*}
  where $\nabla_{12}:L_2\to \widehat{L_1}$ is an isomorphism.
  The map $\phi:L_1\times \widehat{L_1}\to L$ defined by $\phi(x,\chi)=(x,\nabla_{12}\inv(\chi))$ makes $(L,\nabla,L_1)$ standard.
\end{proof}
\begin{lemma}
  \label{lemma:split-maximal-isotropics}
  Let $\nabla$ be a symplectic self-duality of a \lca{} $L$.
  Let $L_1$ be a maximal isotropic subgroup of $L$ which admits a complement $L_2$.
  If multiplication by $2$ is a surjective endomorphism of $\Hom(L_2,\widehat{L_2})$ then $L_1$ admits an isotropic complement.
\end{lemma}
\begin{proof}
  As usual, identify $\hat L$ with $\widehat{L_1} \times \widehat{L_2}$.
  The hypothesis implies that
  \begin{equation*}
    \nabla=\mat 0{\nabla_{12}}{-\widehat{\nabla_{12}}}{\nabla_{22}},
  \end{equation*}
  where $\widehat{\nabla_{22}}=-\nabla_{22}$.
  Since $\nabla$ is an isomorphism, so is $\nabla_{12}$.
  Let $\alpha$ be the automorphism $\alpha:L\to L$ with matrix $\mat{\id{L_1}}X0{\id{L_2}}$ where $X=\widehat{\nabla_{12}}\inv\circ X'$ and $2X'=\nabla_{22}$.
  Then
  \begin{equation}
    \label{eq:38}
    \hat \alpha\circ \nabla \circ \alpha = \mat 0{\nabla_{12}}{-\widehat{\nabla_{12}}}0.
  \end{equation}
  If $L_2'=\alpha(L_2)$, then $L_2'$ is also a complement of $L_1$ and (\ref{eq:38}) implies that $L_2'$ is isotropic.
\end{proof}
\begin{remark}
  \label{remark:split-maximal-isotropics}
  Lemma~\ref{lemma:split-maximal-isotropics} holds under a weaker hypothesis.
  Suppose that $\nabla_{22}=\phi-\hat \phi$ for some $\phi:L_2\to \widehat{L_2}$.
  Then the proof goes through with $X=-\widehat{\nabla_{12}}\inv\circ \phi$.
\end{remark}
Let $p$ be a prime number.
Recall that a group has exponent $p$ if every element other than the identity has order $p$.
Thus \lca s of exponent $p$ are just locally compact $\mathbb{F}_p$-vector spaces.
\begin{theorem}
  \label{theorem:exponent-p}
  If $\nabla$ is a symplectic self-duality of a \lca{} $L$ of exponent $p$ for some prime $p$, then $(L,\nabla)$ is standard.
\end{theorem}
\begin{proof}
  Being a locally compact vector space over a finite field, $L$ has a compact open subgroup (namely the span of a compact open neighborhood of the identity) and so it has a maximal isotropic subgroup $L_1$ by Lemma~\ref{lemma:compact-open-maximal-isotropic}.
  As a vector space $L_1$ splits from $L$, but since the quotient is discrete, it splits as a topological vector space as well.
  Therefore when $p$ is odd, the standardness of $(L,\nabla)$ follows from Lemmas~\ref{lemma:polarization-implies-standard} and~\ref{lemma:split-maximal-isotropics}.
  For $p=2$, denote by $L_2$ any complement of $L_1$.
  Being a discrete abelian group of exponent $p$, $L_2$ is isomorphic to $(\Z/2\Z)^{\oplus \lambda}$ for some cardinal $\lambda$. Pick $I$ such that $\lambda$ is the cardinality of $I$ so that $\nabla_{22}:L_2\rightarrow\widehat{L_2}$ decomposes into $(\nabla_{22})_{ij}$ as explained in the paragraph preceding Lemma~\ref{lemma:polarization-implies-standard}.  Note that $(\nabla_{22})_{ii}=0$ and $(\nabla_{22})_{ij}=-\widehat{(\nabla_{22})_{ji}}$.
  Choose a total ordering of $I$.
  Define $\phi: L_2\to \widehat{L_2}$ by
  \begin{equation*}
    \phi_{ij}=
    \begin{cases}
      (\nabla_{22})_{ij} & \text{ if } i>j\\
      0 & \text{ otherwise}.
    \end{cases}
  \end{equation*}
  Then $\nabla_{22}=\phi-\hat \phi$ and Remark~\ref{remark:split-maximal-isotropics} applies.
\end{proof}
\section{The peeling lemma and some applications}
\label{sec:peeling}
The following lemma, originally due to Ranga Rao \cite[Lemma~13]{MR1280069}, allows one to peel off an isotropic summand and its Pontryagin dual from a self-dual group, thereby decomposing it into the orthogonal sum of a smaller self-dual group and a group of the form $A\times \hat A$.
It will be used to establish the standardness of \lca{s} with symplectic self-duality when the group in question is a finite direct sum of certain kinds of elementary subgroups (see Proposition~\ref{prop:Krull}).
\begin{lemma}
  [Peeling lemma]
  \label{peeeling-lemma}
  Let $\nabla$ be a self-duality of a locally compact abelian group $L$.
  Suppose that $A$ is a split closed isotropic subgroup of $L$.
  Then $(L,\nabla)$ is isomorphic to the orthogonal direct sum of $(A\times \hat A,\nabla')$ and $(L_A,\nabla_A)$, where $\nabla'$ is a symplectic self-duality of $A\times \hat A$ for which $A$ is isotropic and $(L_A,\nabla_A)$ is as in Lemma~\ref{lemma:subquotient}.
\end{lemma}
\begin{proof}
  Suppose that $L=A\times B$.
  Then $\nabla\inv(A^\perp)=A\times (\nabla\inv(A^\perp)\cap B)$.  Note that
  $\hat L$ inherits the decomposition $A^\perp\times B^\perp$, so that $L=A\times \nabla\inv(B^\perp)\times (\nabla\inv(A^\perp)\cap B)$.
  Now $B^\perp$ is isomorphic to $\hat A$.
  Furthermore, it is clear that the first two factors are orthogonal to the third, which is isomorphic to $L_A$ in such a way that the induced self-duality on it is $\nabla_A$.
\end{proof}
\begin{defn*}
  [Universally isotropic group]
  A \lca{} $L$ is said to be universally isotropic if there exists no non-trivial homomorphism $\nabla:L\to \hat L$ such that $\nabla(x)(x)=0$ for all $x\in L$.
\end{defn*}
Universally isotropic subgroups are isotropic for any \lca{} with symplectic self-duality.
\begin{defn*}
  [Topologically locally cyclic group]
  A topological group is said to be topologically locally cyclic if it has a dense locally cyclic subgroup (a group is locally cyclic is every finitely generated subgroup is cyclic).
\end{defn*}
Besides cyclic groups, $\R$, $\Q_p$, $\Z_p$, $\Q_p/\Z_p$, $\Q$, $\Q/\Z$ and $\T$ are some examples of topologically locally cyclic groups.
\begin{lemma}
  \label{lemma:locally-cyclic}
  Topologically locally cyclic \lca s are universally isotropic.
\end{lemma}
\begin{proof}
  Let $E$ be the locally cyclic dense subgroup of $L$.
  Given any $x,y\in E$, since the group generated by $x$ and $y$ is cyclic, there exists $z\in E$ such that $x=mz$ and $y=nz$ for some integers $m$ and $n$.
  Consequently, $\nabla(x)(y)=mn\nabla(z)(z)=0$.
  Since $E\times E$ is dense in $L\times L$ and $\nabla$ is continuous, $\nabla\equiv 0$.
\end{proof}
\begin{remark}
  Lemma~\ref{lemma:locally-cyclic} implies that a locally cyclic group can not have a symplectic self-duality even if it is self-dual.
\end{remark}
\begin{prop}
  \label{prop:Krull}
  Let $\{L_i\}$ be a collection of universally isotropic \lca s such that each $L_i$ is either $2$\dash regular or has a universally isotropic Pontryagin dual.
  Let $\mathfrak C$ be a class of locally compact abelian groups such that each $L\in \mathfrak C$ is isomorphic to a unique finite direct sum of the groups $L_i$ or $\widehat{L_i}$, and such that if $L=L'\times L''$ with $L,L'\in \mathfrak C$, then $L''\in \mathfrak C$.
  If $\nabla$ is a symplectic self-duality of $L\in \mathfrak C$ then $(L,\nabla)$ is standard.
\end{prop}
\begin{proof}
  For $L\in \mathfrak C$, define the rank of $L$ to be the number of summands in the decomposition of $L$ into groups of the form $L_i$ or $\widehat{L_i}$.
  Proposition~\ref{prop:Krull} is proved by induction on the rank of $L$.
  If $L$ has rank $0$, there is nothing to prove.
  Otherwise, the self-duality of $L$ and the uniqueness of its decomposition ensure that $L$ has a summand of the form $L_i$.
  Now apply the peeling lemma (Lemma~\ref{peeeling-lemma}) to $L$ with $A=L_i$.
  Lemmas~\ref{lemma:polarization-implies-standard} and~~\ref{lemma:split-maximal-isotropics} imply that the $A\times \hat A$ factor is standard.
  $L_A$ is also in $\mathfrak C$ and has rank two less than that of $L$, and is therefore standard by the induction hypothesis.
\end{proof}
\begin{cor}
  \label{cor:fdvs}
  Let $F$ be one of the topological fields $\R$, $\Q_p$ or $\Z/p\Z$ ($p$ any prime number).
  Let $\nabla$ be a symplectic self-duality of the additive group $L$ of a finite dimensional vector space over $F$.
  Then $(L,\nabla)$ is standard.
\end{cor}
\begin{proof}
  As $L$ is isomorphic to a unique finite direct sum of one-dimensional subspaces which are topologically locally cyclic, Proposition~\ref{prop:Krull} applies.
\end{proof}
\begin{defn*}
  [Topological torsion group, topological $p$\dash group]
  A \lca{} $G$ is said to be a topological torsion group (topological $p$\dash group) if it has a compact open subgroup $G$ such that $\hat G$ and $L/G$ are torsion groups (respectively, $p$\dash torsion groups, i.e., every element is annihilated by some power of $p$).
\end{defn*}
\begin{defn*}
  [Finite $p$\dash rank]
  \label{defn:finite-p-rank}
  A topological $p$\dash group $L$ is said to have finite $p$\dash rank if it has a compact open subgroup $G$ such that $\hat G$ and $L/G$ have finite $p$\dash rank (an abelian group is said to have finite $p$\dash rank if its subgroup of elements annihilated by $p$ is finite).
\end{defn*}
\begin{cor}
  \label{cor:finite-p-rank}
  If $\nabla$ is a symplectic self-duality of $L$, which is a direct sum over a finite number of primes $p$ of groups of finite $p$\dash rank then $(L,\nabla)$ is standard.
\end{cor}
\begin{proof}
  Our $L$ is isomorphic to a unique finite direct sum of the topologically locally cyclic groups $\Q_p$, $\Z_p$, $\Q_p/\Z_p$ and finite cyclic $p$\dash groups, where $p$ ranges over the finite set of prime numbers above (see \cite[Lemma~2.8 and Proposition~2.9]{MR2329311} and \cite[Theorem~13]{MR0035776}). Therefore Proposition~\ref{prop:Krull} applies.
\end{proof}
\begin{cor}
  \label{cor:finite}
  If $\nabla$ is a symplectic self-duality of a finite abelian group $L$, then $(L,\nabla)$ is standard.
\end{cor}
\begin{proof}
  This is a special case of Corollary~\ref{cor:finite-p-rank}.
\end{proof}
\section{Reduction to residual groups}
\label{sec:red-open-compact}
The Pontryagin-van Kampen theorem states that every locally compact abelian group can be written as $\R^n\times E$ for a unique non-negative integer $n$ and some \lca{} $E$ which has a compact open subgroup.
$E$ is completely determined by $L$ up to isomorphism \cite[Corollary~1]{MR0578649}.
\begin{defn*}
  [Residual group]
  For a \lca{} $L$, let $B(L)$ denote the subgroup of compact elements (elements which are contained in a compact subgroup of $L$), and let $D(L)$ denote the maximal divisible subgroup of $L$.
  $L$ is said to be residual if $D(L)\subset B(L)$ and $D(\hat L)\subset B(\hat L)$.
\end{defn*}
A refinement of the Pontryagin-van Kampen theorem is Robertson's structure theorem \cite[Theorem~2]{MR0578649}:
\begin{theorem}
  \label{theorem:Robertson}
  Every \lca{} $L$ can be written as
  \begin{equation}
    \label{eq:1}
    L\cong\R^n\times \Q^{\oplus \lambda}\times \hat \Q^\mu \times E
  \end{equation}
  where $n$ is a non-negative integer, $\lambda$ and $\mu$ are cardinal numbers, and $E$ is a residual \lca{}.  The group
  $L$ determines $n$, $\lambda$, $\mu$ and the isomorphism class of $E$ uniquely.
\end{theorem}
It is clear from the above result that $L$ is self-dual if and only if $\lambda=\mu$ and $E$ is self-dual.
\begin{theorem}
  \label{theorem:reduction-to-residual}
  Let $L$ be a \lca{} with decomposition as in Theorem~\ref{theorem:Robertson}.
  Every symplectic self-duality $\nabla$ of $L$ gives rise to a symplectic self-duality $\nabla'$ of $E$.  The pair
  $(L,\nabla)$ is  standard  if $(E,\nabla')$ is a standard pair.
\end{theorem}
\begin{proof}
  If $\nabla$ is a symplectic self-duality of $L$, using the peeling lemma and induction on $n$, one may reduce to the case where $n=0$.
  The factor $\hat \Q^\mu$ is connected and compact, with discrete torsion-free Pontryagin dual.
  Therefore, there are no non-trivial homomorphisms from $\hat \Q^\mu$ to its Pontryagin dual.
  Therefore $\hat \Q^\mu$ is an isotropic factor.
  By the peeling lemma, $(L,\nabla)$ is the orthogonal sum of a symplectic self-duality on $\Q^{\oplus \mu}\times \hat \Q^\mu$ for which $\hat \Q^\mu$ is isotropic (which is standard by Lemmas~\ref{lemma:polarization-implies-standard} and~\ref{lemma:split-maximal-isotropics}) and a self-dual group with no $\hat \Q^\mu$\dash part, and hence, by self-duality also no $\Q^\lambda$\dash part, namely a residual group.
  This group is isomorphic to $E$ by the uniqueness part of Robertson's structure theorem (Theorem~\ref{theorem:Robertson}).
  Thus $(L,\nabla)$ is an orthogonal sum of a standard pair and $(E,\nabla')$, where $\nabla'$ is a symplectic self-duality of $E$.
  If $(E,\nabla')$ is standard, then $(L,\nabla)$, being an orthogonal sum of standard pairs, is a standard pair.
\end{proof}
\begin{remark}
  \label{remark:residual}
  In view of Theorem~\ref{theorem:reduction-to-residual}, Theorem~\ref{theorem:exponent-p} and Corollaries~\ref{cor:fdvs}--\ref{cor:finite} continue to hold when the hypotheses are applied only to the residual part of $L$.
\end{remark}

\section{Duality of extensions}
\label{sec:antidual}
Let $G$ be a compact abelian group and suppose that $L$ is an extension over the discrete group $\hat G$ with kernel $G$:
\begin{equation*}
  \label{eq:5}
  \xses {\xi}GL{\hat G}fq.
\end{equation*}
There is exactly one topology on $L$ under which $f$ and $q$ are continuous, the one where a system of open neighborhoods of $0$ is generated by the open neighborhoods of $0$ in $G$.
Thus, the algebraic structure of the extension determines it topologically.
Equivalence classes of such extensions are in bijective correspondence with elements of the Baer group $E(\hat G,G)$ (see \cite[XIV.1]{homologicalalgebra}).
Applying Pontryagin duality to \ref{eq:5} gives another extension of the same type
\begin{equation}
  \label{eq:2}
  \xses {\hat \xi}GL{\hat G}{\hat q}{\hat f}.
\end{equation}
Thus, Pontryagin duality induces an involution $\xi\mapsto \hat \xi$ of $E(\hat G,G)$.
\begin{defn*}
  [{\cite[Definition~3.3]{MR1475126}}]
  The extension \ref{eq:5} is said to be autodual (antidual) if $\hat \xi=\xi$ (respectively, $\hat \xi =-\xi$).
\end{defn*}
\begin{remark}
  \label{remark:anti-dual-extensions}
  If \ref{eq:5} is an autodual or antidual extension, then there exists an isomorphism $\nabla:L\to \hat L$ such that
  \begin{equation}
    \label{eq:3}
    \xymatrix{
      (\xi) & 0 \ar[r] & G \ar@{=}[d] \ar[r]^f & L \ar[d]^\nabla \ar[r]^q & \hat G \ar[d]^\epsilon \ar[r] & 0\\
      (\hat \xi) & 0 \ar[r] & G \ar[r]_{\hat q} & \hat L \ar[r]_{\hat f} & \hat G \ar[r] & 0
    }
  \end{equation}
  is a commutative diagram, with $\epsilon=\id{\hat G}$ in the autodual case, and $\epsilon=-\id{\hat G}$ in the antidual case.
\end{remark}
\begin{lemma}
  \label{lemmma:max-iso}
  Let $\nabla$ be a symplectic self-duality of a \lca{} $L$ and suppose that $G$ is a compact open subgroup of $L$ for which $\nabla\inv(G^\perp)=G$, i.e., $G$ is maximal isotropic.
  When $f:G\to L$ denotes the inclusion map and $q:L\to \hat G$ is defined by $q(x)(g)=\nabla(f(g))(x)$ for all $x\in L$ and $g\in G$, (\ref{eq:3}) has exact rows and commutes (with $\epsilon=-\id{\hat G}$).
\end{lemma}
\begin{proof}
  Since $G=\nabla\inv(G^\perp)$, the kernel of $q$ is $G$.
  Since any character of $G$ can be extended to $L$, and $\nabla:L\to \hat L$ is an isomorphism, $q$ is surjective.
  It remains to check that (\ref{eq:3}) commutes (with $\epsilon=-1$).
  Indeed, for $g\in G$ and $x\in L$,
  \begin{eqnarray*}
    \hat q(g)(x) & = & q(x)(g)\\
    & = & \nabla(f(g))(x),
  \end{eqnarray*}
  showing that the box on the left commutes.
  Also, for $x\in L$ and $g\in G$,
  \begin{eqnarray*}
    \hat f (\nabla(x))(g) & = & \nabla(x)(f(g))\\
    & = & -\nabla(f(g))(x)\\
    & = & -q(x)(g),
  \end{eqnarray*}
  showing that the box on the right commutes.
\end{proof}
\begin{lemma}
  \label{lemma:compact-open-maximal-isotropic}
  Let $\nabla$ be a symplectic self-duality of a \lca{} $L$.
  If $L$ has a compact open subgroup, then $L$ has a compact open maximal isotropic subgroup.
\end{lemma}
\begin{proof}
  According to \cite[Proposition~4.6]{MR1651168}, $L$ has a compact open subgroup $G$ such that $G\subset \nabla\inv(G^\perp)$ and $\nabla\inv(G^\perp)/G$ is finite.
  If $G\neq \nabla\inv(G^\perp)$, pick any $x\in \nabla\inv(G^\perp)\setminus G$.
  Let $G'$ be the subgroup of $L$ generated by $G$ and $x$; it is again compact and open.
  Since $\nabla$ is a symplectic self-duality, $\nabla(x)(x)=0$.
  It follows that $G\subsetneq G'\subset \nabla\inv(G^{\prime \perp})\subset \nabla\inv(G^\perp)$.
  Thus, by replacing $G$ with $G'$, the index of $G$ in $\nabla\inv(G^\perp)$ can be reduced.
  In a finite number of steps, a compact open subgroup $G$ is obtained for which $G=\nabla\inv(G^\perp)$, so $G$ is a maximal isotropic subgroup.
\end{proof}
\begin{theorem}
  \label{theorem:anti-dual-extensions}
  Let $L$ be a locally compact abelian group.
  Consider the following statements:
  \begin{enumerate}
  \item \label{item:1} The group $L$ has a compact open subgroup and admits a symplectic self-duality.
  \item \label{item:2} There exists a compact abelian group $G$ such that $L$ is an antidual extension over $\hat G$ with kernel $G$.
  \end{enumerate}
  Then \ref{item:1} implies \ref{item:2}.
  If $L$ is $2$\dash regular, then \ref{item:1} and \ref{item:2} are equivalent.
\end{theorem}
\begin{proof}[Proof that \ref{item:1} implies \ref{item:2}]
  By Lemma~\ref{lemma:compact-open-maximal-isotropic}, $L$ has a compact open maximal isotropic subgroup $G$.
  By Lemma~\ref{lemmma:max-iso} and Remark~\ref{remark:anti-dual-extensions}, $L$ is an antidual extension over $\hat G$ with kernel $G$.
\end{proof}
\begin{proof}[Proof that \ref{item:2} implies \ref{item:1}]
  Assume that $L$ is an antidual extension over $\hat G$ with kernel $G$.
  By Remark~\ref{remark:anti-dual-extensions}, there exists an isomorphism $\nabla:L\to \hat L$ such that (\ref{eq:3}) (with $\epsilon=-\id{\hat G}$) commutes.
  Dualizing (\ref{eq:3}) gives
  \begin{equation*}
    \xymatrix{
      0 \ar[r] & G \ar[r]^f \ar[d]_{-\id{G}} & L \ar[r]^q \ar[d]_{\hat \nabla} & \hat G \ar@{=}[d]  \ar[r] & 0\\
      0 \ar[r] & G \ar[r]_{\hat q} & \hat L \ar[r]_{\hat f} & \hat G \ar[r] & 0.
    }
    \end{equation*}
  The commutativity of the above diagram is equivalent to the commutativity of
  \begin{equation*}
    \xymatrix{
      0 \ar[r] & G \ar@{=}[d] \ar[r]^f & L \ar[d]^{-\hat \nabla} \ar[r]^q & \hat G \ar[d]^{-\id{\hat G}} \ar[r] & 0\\
      0 \ar[r] & G \ar[r]_{\hat q} & \hat L \ar[r]_{\hat f} & \hat G \ar[r] & 0.
    }
    \end{equation*}
  Therefore (\ref{eq:3}) continues to commute when $\nabla$ is replaced by $-\hat \nabla$, and hence also when it is replaced by $(\nabla -\hat \nabla)/2$, which is a symplectic self-duality.
\end{proof}
\section{Extensions and products}
\label{sec:extensions-products}
Suppose that $G=\prod_i G_i$ (a possibly infinite product).
Starting with the extension \ref{eq:5} of Section~\ref{sec:antidual}, construct the diagram
\begin{equation}
  \label{eq:4}
  \xymatrix{
    (\xi) & 0 \ar[r] & G \ar[r]^f \ar[d]^{p_i} \ar@{}[dr]|\square & L \ar[r]^q \ar[d] & \hat G \ar[r] \ar@{=}[d] & 0\\
    (\xi_{i*}) & 0 \ar[r] & G_i \ar[r]^{f_{i*}} \ar@{=}[d] & L_{i*} \ar[r]^{q_{i*}} \ar@{}[dr]|\square & \hat G \ar[r] & 0\\
    (\xi_{ij}) & 0 \ar[r] & G_i \ar[r]^{f_{ij}} & L_{ij} \ar[u] \ar[r]^{q_{ij}} & \hat G_j \ar[r] \ar[u]_{\hat p_j}& 0
  }
\end{equation}
by first constructing $L_{i*}$ as a push-out of the square on the top-left, where $p_i$ denotes the projection onto the $i$th component, completing the second row using the universal property of the the push-out, constructing $L_{ij}$ as the pull-back in the square on the bottom right and completing the third row using the universal property of pull-backs.  Note that
$L$ (together with the map $q:L\to \hat G$) can be recovered from the extensions $\xi_{i*}$ as the fibred product of the $q_{i*}:L_{i*}\to \hat G$'s.
Each $L_{i*}$ (together with the maps $f_{i*}:G_i\to L_{i*}$) can in turn be recovered from the extensions $\xi_{ij}$ as the fibred coproduct of the $f_{ij}:G_i\to L_{ij}$'s.
In fact, $\xi\mapsto (\xi_{ij})$ is the standard isomorphism \cite[VI,~Proposition~1.2]{homologicalalgebra}:
\begin{equation*}
  E(\hat G,G)\tilde\to \prod_{i,j} E(\hat G_j,G_i).
\end{equation*}
The same family of extensions is obtained by first taking a pull-back and then a push-out:
\begin{equation}
  \label{eq:6}
  \xymatrix{
    (\xi) & 0 \ar[r] & G \ar[r]^f \ar@{=}[d] & L \ar[r]^q \ar@{}[dr]|\square & \hat G \ar[r] & 0\\
    (\xi_{*j}) & 0 \ar[r] & G \ar[r]^{f_{*j}}\ar[d]^{p_i} \ar@{}[dr]|\square & L_{*j} \ar[u] \ar[r]^{q_{*j}}\ar[d] & \hat G_j \ar[u]_{\hat p_j}\ar@{=}[d] \ar[r] & 0\\
    (\xi_{ij}) & 0 \ar[r] & G_i \ar[r]^{f_{ij}} & L_{ij} \ar[r]^{q_{ij}} & \hat G_j \ar[r] & 0
  }
\end{equation}
\begin{lemma}
  \label{lemma:dual-product}
  The extensions $\widehat{\xi_{ij}}$ and $(\hat\xi)_{ji}$ are equivalent.
\end{lemma}
\begin{proof}
  Applying Pontryagin duality to (\ref{eq:4}) is the same as starting with $\hat \xi$ and applying the construction of (\ref{eq:6}) to it, but with $i$ and $j$ interchanged.
  The former results in $\widehat{\xi_{ij}}$, while the latter results in $(\hat \xi)_{ji}$.
\end{proof}
\begin{lemma}
  Suppose that $\xi$ is autodual (antidual) and there exists $\nabla:L\to \hat L$ such that (\ref{eq:3}) commutes with $\epsilon=\id{\hat G}$ (respectively, $\epsilon=\id{\hat G}$) and such that $\hat \nabla=\nabla$ (respectively, $\hat \nabla=-\nabla$).
  Then $(\hat\xi)_{ij}=\xi_{ij}$ (respectively, $(\hat\xi)_{ij}=-\xi_{ij}$) in $E(\hat G_j,G_i)$.
\end{lemma}
\begin{proof}
  The constructions of (\ref{eq:4}) and (\ref{eq:6}) are functorial, and when applied to both rows of (\ref{eq:3}), give rise to equivalences $\nabla_{ij}:L_{ij}\to \widehat{L_{ij}}$.
\end{proof}
\section{Standard triples viewed as extensions}
\label{sec:standard}
Let $A$ be a locally compact abelian group which has a compact open subgroup $B$.
Then $B^\perp$ is a compact open subgroup of $\hat A$.
Consider the standard triple $(L^0,\nabla^0,G^0)$ as in Example~\ref{eg:main-questions}, where $L^0=A\times \hat A$ and $G^0=B\times B^\perp$.
Since $G^0$ is compact open, we may think of $L^0$ as an extension
\begin{equation}
  \label{eq:7}
  \xses{\xi^0}{G^0}{L^0}{\widehat{G^0}}{f^0}{q^0}
\end{equation}
by Lemma~\ref{lemmma:max-iso}.
Applying the construction of (\ref{eq:4}) to the product decomposition $G=B\times B^\perp$ gives four extensions $\xi^0_{ij}$, $i,j\in \{1,2\}$.
Keeping track of the morphisms functorially induced by $\nabla^0$ one gets
\begin{equation}
  \label{eq:24}
  \tag{$\xi^0_{12}$}
  \xymatrix{
    0 \ar[r] & B \ar[r]^{f^0_{12}} \ar@{=}[d] & A \ar[r]^{q^0_{12}} \ar[d]^{\nabla^0_{12}} & \widehat{B^\perp} \ar[d]^{-1} \ar[r] & 0\\
    0 \ar[r] & B \ar[r]_{\widehat{q^0_{21}}} & A \ar[r]_{\widehat{f^0_{21}}} & \widehat{B^\perp} \ar[r] & 0,
  }
\end{equation}
where $f^0_{12}$ is the inclusion map $B\to A$, $q^0_{12}$ is the restriction of an element of $A$ (a character of $\hat A$) to $B^\perp$ and $\nabla^0_{12}$ takes an element of $A$ to its inverse,
\begin{equation}
  \label{eq:25}
  \tag{$\xi^0_{21}$}
  \xymatrix{
    0 \ar[r] & B^\perp \ar[r]^{f^0_{21}} \ar@{=}[d] & \hat A \ar[r]^{q^0_{21}} \ar[d]^{\nabla^0_{21}} & \hat B \ar[d]^{-1} \ar[r] & 0\\
    0 \ar[r] & B^\perp \ar[r]_{\widehat{q^0_{12}}} & \hat A \ar[r]_{\widehat{f^0_{12}}} & \hat B \ar[r] & 0,
  }
\end{equation}
where $f^0_{21}$ is the inclusion map $B^\perp\to \hat A$, $q^0_{21}$ takes a character of $A$ to the inverse of its restriction to $B$ and $\nabla^0_{21}$ is the identity map,
\begin{equation}
  \label{eq:36}
  \tag{$\xi^0_{11}$}
  \xymatrix{
    0 \ar[r] & B \ar[r]^{f^0_{11}} \ar@{=}[d] & B\times \hat B \ar[d]^{\nabla^0_{11}} \ar[r]^{q^0_{11}} & \hat B \ar[r] \ar[d]^{-1}& 0\\
    0 \ar[r] & B \ar[r]_{\widehat{q^0_{11}}} & \hat B \times B \ar[r]_{\widehat{f^0_{11}}} & \hat B \ar[r] & 0,
  }
\end{equation}
where $f^0_{11}(x)=(x,0)$, $q^0_{11}(x,\chi)=-\chi$ and $\nabla^0_{11}(x,\chi)=(\chi,-x)$ for $x\in B$ and $\chi \in \hat B$, and finally,
\begin{equation}
  \label{eq:37}
  \tag{$\xi^0_{22}$}
  \xymatrix{
    0 \ar[r] & B^\perp \ar[r]^{j^0_{22}} \ar@{=}[d] & \widehat{B^\perp}\times B^\perp \ar[d]^{\nabla^0_{22}} \ar[r]^{q^0_{22}} & \widehat{B^\perp} \ar[d]^{-1} \ar[r] & 0\\
    0 \ar[r] & B^\perp \ar[r]_{q^0_{22}} & B^\perp\times \widehat{B^\perp} \ar[r]_{\widehat{q^0_{22}}} & \widehat{B^\perp} \ar[r] & 0.
  }
\end{equation}
where $f^0_{22}(\chi)=(0,\chi)$, $q^0_{22}(x,\chi)=x$ and  $\nabla^0_{22}(x,\chi)=(\chi,-x)$ for $x\in \widehat{B^\perp}$, $\chi \in B^\perp$.
\emph{In particular, $\xi^0_{11}$ and $\xi^0_{22}$ are split.}
\begin{theorem}
  \label{theorem:standard-triples}
  Let $\nabla$ be a symplectic self-duality of a locally compact abelian group $L$ which has a compact open subgroup.
  Let $G$ be a compact open maximal isotropic subgroup of $L$.
  Consider the following statements:
  \begin{enumerate}
  \item \label{item:3}The triple $(L,\nabla,G)$ is  standard.
  \item \label{item:4}The group $G$ admits a decomposition $G=G_1\times G_2$ such that the extensions $\xi_{11}$ and $\xi_{22}$ obtained by applying (\ref{eq:4}) to the extension $\xi$ constructed in Lemma~\ref{lemmma:max-iso} are split.
  \end{enumerate}
  Then \ref{item:3} implies \ref{item:4}.
  If $G$ is $2$\dash regular then \ref{item:3} and \ref{item:4} are equivalent.
\end{theorem}
\begin{proof}
  That \ref{item:3} implies \ref{item:4} is clear from the discussion preceding the statement of Theorem~\ref{theorem:standard-triples}.

  For the converse, set $A=L_{12}$ and $B=G_1$.  Observe that
  $B$ is a subgroup of $A$ via the inclusion $f_{12}$.
  Thus $A/B=\widehat{B^\perp}$ is identified with $\widehat{G_2}$, hence $B^\perp$ with $G_2$.
  Let $\xi$ denote the extension associated to the triple $(L,\nabla,G)$ by Lemma~\ref{lemmma:max-iso}.
  An isomorphism $\phi:L^0\to L$ with the required properties will be constructed by constructing isomorphisms $\phi_{ij}:L^0_{ij}\to L_{ij}$ which give equivalences of extensions between $\xi^0_{ij}$ and $\xi_{ij}$ for all $i,j\in \{1,2\}$.

  Let $\phi_{12}:A\to L_{12}$ be the identity map, $\phi_{21}:\hat A\to L_{21}$ be $\nabla_{21}\inv \widehat{\phi_{12}}\inv$.
  Since the extensions $\xi_{11}$ and $\xi_{22}$ are split and $G$ is $2$\dash regular, Lemmas~\ref{lemma:polarization-implies-standard} and~\ref{lemma:split-maximal-isotropics} imply that $(L_{ii},\nabla_{ii},G_i)$ are standard, giving isomorphisms $\phi_{ii}:L^0_{ii}\to L_{ii}$ which give equivalences of the extensions $\xi^0_{ii}$ and $\xi_{ii}$ for $i\in \{1,2\}$.

  The isomorphisms $\phi_{ij}$ above have been chosen in such a way that
  \begin{equation*}
    \xymatrix{
      L^0_{ij}\ar[r]^{\nabla^0_{ij}} \ar[d]_{\phi_{ij}} & (\widehat{L^0})_{ij}\\
      L_{ij} \ar[r]^{\nabla_{ij}} \ar[r] & (\hat L)_{ij} \ar[u]_{\widehat{\phi_{ji}}}
    }
  \end{equation*}
  commutes.
  Recall that $\xi$ can be recovered from the $\xi_{ij}$'s using fibred products and coproducts (Section~\ref{sec:extensions-products}).
  By the universal properties of fibred products and coproducts, there is an induced map $\phi:L^0\to L$, which gives an isomorphism between $(L^0,\nabla^0,G^0)$ and $(L,\nabla,G)$.
\end{proof}
\section{Extensions of finite $p$-groups, non-standard triples}
\label{sec:finite}
For each finite non-increasing sequence $(\lambda_1,\ldots,\lambda_l)$ of positive integers let $p^\lambda$ denote the function $(x_1,\ldots,x_l)\mapsto (p^{\lambda_1}x_1,\ldots,p^{\lambda_l}x_l)$.
Let
\begin{gather*}
  \dtorsion \lambda = \{\mu\in \T^l\;:\;p^\lambda(\mu)=0\},\\
  \ctorsion \lambda = \Z^l/p^\lambda(\Z^l).
\end{gather*}
Note that $\dtorsion \lambda$ and $\ctorsion \lambda$ are mutual Pontryagin duals under the pairing
\begin{equation*}
  ((\mu_1,\ldots,\mu_l),(x_1,\ldots,x_l))\mapsto \mu_1x_1+\cdots +\mu_lx_l.
\end{equation*}
To the projective resolution
\begin{equation*}
  \ses{\Z^l}{\Z^l}{\dtorsion \lambda}{p^\lambda}{p^{-\lambda}}
\end{equation*}
apply the contravariant functor $\Hom(-,\ctorsion \lambda)$ to get the exact sequence
\begin{equation}
  \label{eq:13}
  \xymatrix{
    \Hom(\Z^l,\ctorsion \lambda) \ar[r]^{\circ p^\lambda} & \Hom(\Z^l,\ctorsion \lambda) \ar[r]^{\partial} & E(\dtorsion \lambda, \ctorsion \lambda) \ar[r] & 0.
  }
\end{equation}
To an $l\times l$ matrix $\alpha=(\alpha_{ij})$ with integer entries, associate the homomorphism $\Z^l\to \ctorsion \lambda$ which takes the $i$th coordinate vector $e_i$ to the image of $(\alpha_{i1},\ldots,\alpha_{il})$ in $\ctorsion \lambda$.
Its image $\partial(\alpha)$ in $E(\dtorsion \lambda,\ctorsion \lambda)$ is determined by the values of $\alpha_{ij}$ modulo $p^{\min\{\lambda_i,\lambda_j\}}$ and the extension $\xi(\alpha)$ corresponding to $\partial(\alpha)$ is given by the push-out
\begin{equation}
  \label{eq:22}
  \xymatrix{
    &0 \ar[r] & \Z^l \ar[r]^{p^\lambda} \ar[d]_{\alpha} \ar@{}[dr]|\square & \Z^l \ar[r]^{p^{-\lambda}}\ar[d]^\beta & \dtorsion\lambda \ar[r]\ar@{=}[d] & 0\\
    (\xi(\alpha))&0 \ar[r] & \ctorsion\lambda \ar[r]_f & L \ar[r]_q & \dtorsion\lambda \ar[r] & 0
  }
\end{equation}
with the map $q:L\to \dtorsion\lambda$ determined by the universal property of the push-out and the maps
\begin{equation*}
  \xymatrix{
    \Z^l \ar[r]^{p^\lambda}\ar[d]_{\alpha} & \Z^l \ar[d]^{p^{-\lambda}}\\
    \ctorsion\lambda \ar[r]_0 & \dtorsion\lambda.
  }
\end{equation*}
In what follows, we will always replace the entries of $\alpha$ in such a way that $0\leq \alpha_{ij}<p^{\min\{\lambda_i,\lambda_j\}}$.
\begin{remark}
  From the explicit construction of the push-out in (\ref{eq:22}), one sees that the group $L$ is the quotient of $\Z^{2l}$ by the column space of the block matrix
  $\mat{p^\lambda}{-\alpha}0{p^\lambda}\in M_{2l}(\Z)$.
  After left and right multiplication by elements of $GL_{2l}(\Z)$ this matrix can be reduced to Smith canonical form, namely a diagonal matrix with entries $p^{\mu_1},p^{\mu_2},\ldots,p^{\mu_{2l}}$, with $\mu_1\geq \mu_2\geq \cdots \geq \mu_{2l}\geq 0$.
  Therefore $L$ is isomorphic to $\Z/p^{\mu_1}\times \cdots\times \Z/p^{\mu_{2l}}$.
  The rank of $L$ is $2l$ less the number of $\mu_i$ which are $0$, or in other words, $2l-\rank_{\Z/p\Z}(\bar \alpha)$, where $\bar \alpha$ is the image of $\alpha$ in $M_l(\Z/p\Z)$.
\end{remark}
Starting with the injective resolution
\begin{equation*}
  \ses{\ctorsion\lambda}{\T^l}{\T^l}{\widehat{p^{-\lambda}}}{\widehat{p^\lambda}}
\end{equation*}
apply the covariant functor $\Hom(\dtorsion\lambda,-)$ to get the exact sequence
\begin{equation*}
  \xymatrix{
    \Hom(\dtorsion\lambda,\T^l)\ar[r]^{\widehat{p^\lambda}\circ} & \Hom(\dtorsion\lambda,\T^l)\ar[r]^{\partial^*} & E(\dtorsion\lambda,\ctorsion\lambda) \ar[r] & 0.
  }
\end{equation*}
Starting with $\hat \alpha:\dtorsion \lambda\to \T^l$ (the adjoint of $\alpha:\Z^l\to \ctorsion\lambda$ described above), an extension $\xi^*(\alpha)$ with class $\partial^*(\alpha)$ is constructed as a pullback
\begin{equation}
  \label{eq:23}
  \xymatrix{
    (\xi^*(\alpha))&0 \ar[r] & \ctorsion\lambda \ar[r]^{f^*} \ar@{=}[d] & L^* \ar[r]^{q^*} \ar[d] \ar@{}[dr]|\square & \dtorsion\lambda \ar[r] \ar[d]^{\hat \alpha} & 0\\
    &0 \ar[r] & \ctorsion\lambda \ar[r]_{\widehat{p^{-\lambda}}} & \T^l \ar[r]_{\widehat{p^\lambda}} & \T^l \ar[r] & 0.
  }
\end{equation}
Since (\ref{eq:23}) is the Pontryagin dual of (\ref{eq:22}),
\begin{equation}
  \label{eq:8}
  \xi^*(\alpha)=\widehat{\xi(\alpha)}.
\end{equation}
\begin{prop}
  \label{prop:transpose}
  For any $l\times l$ matrix $\alpha$ with integer entries $\xi^*(\transpose\alpha)=\xi(\alpha)$.
\end{prop}
\begin{proof}
  Start with (\ref{eq:22}).
  Since $\Z^l$ is projective, there exists $\tilde\gamma$ such that the diagram
  \begin{equation*}
    \xymatrix{
      & \Z^l \ar@{.>}[dl]_{\tilde \gamma}\ar[d]^{\widehat{\transpose\alpha}\circ p^{-\lambda}} &\\
      \T^l \ar[r]_{\widehat{p^\lambda}} & \T^l \ar[r] & 0
    }
  \end{equation*}
  commutes.
  This choice of $\tilde \gamma$ ensures that the diagram
  \begin{equation*}
        \xymatrix{
      \Z^l \ar[r]^{p^\lambda} \ar[d]_\alpha & \Z^l \ar[d]^{\tilde \gamma}\\
      \ctorsion\lambda \ar[r]_{\widehat{p^{-\lambda}}} & \T^l
    }
  \end{equation*}
  commutes.
  By the universal property of the push-out, there exists a morphism $\gamma:L\to \T^l$ such that the diagram
    \begin{equation}
    \label{eq:33}
    \xymatrix{
      0 \ar[r] & \Z^l \ar[r]^{p^\lambda} \ar[d]_\alpha \ar@{}[dr]|\square & \Z^l \ar[r]^{p^{-\lambda}} \ar[d]^\beta & \dtorsion\lambda \ar@{=}[d] \ar[r] & 0\\
      0 \ar[r] & \ctorsion\lambda \ar[r]^f \ar@{=}[d] & L \ar[r]^q \ar[d]_\gamma\ar@{}[dr]|\square & \dtorsion\lambda \ar[r] \ar[d]^{\widehat{\transpose\alpha}}& 0\\
      0 \ar[r] & \ctorsion\lambda \ar[r]_{\widehat{p^{-\lambda}}} & \T^l \ar[r]_{\widehat{p^\lambda}} & \T^l \ar[r] & 0
    }
  \end{equation}
  commutes.
  We will show that the square on the bottom right is indeed Cartesian.
  Let $L'$ be the pullback
  \begin{equation*}
    \xymatrix{
      L' \ar[r] \ar[d] \ar@{}[dr]|\square&\dtorsion\lambda \ar[d]^{\widehat{\transpose\alpha}}\\
      \T^l \ar[r]^{\widehat{p^\lambda}} & \T^l
    }
  \end{equation*}
  Comparison with (\ref{eq:23}) shows that $L'$ is an extension $\xi^*(\transpose\alpha)$ over $\dtorsion\lambda$ with kernel $\ctorsion\lambda$.
  By the universal property of pullbacks, there is a morphism $L\to L'$ which is in fact a morphism of extensions, hence an isomorphism, showing that the square on the bottom right of (\ref{eq:33}) is Cartesian.
  The top half of (\ref{eq:33}) shows that the extension in its second row is $\xi(\alpha)$, while the bottom half shows that the same extension is $\xi^*(\transpose\alpha)$.
\end{proof}
\begin{theorem}
  \label{theorem:transpose}
  For any $l\times l$\dash matrix $\alpha$ with integer entries, $\widehat{\xi(\alpha)}=\xi(\transpose\alpha)$.
\end{theorem}
\begin{proof}
  This is an immediate consequence of (\ref{eq:8}) and Proposition~\ref{prop:transpose}.
\end{proof}
\begin{cor}
  \label{cor:antidual-matrix}
  For any $l\times l$\dash matrix $\alpha$ with integer entries, $\xi(\alpha)$ is autodual (antidual) if and only if $\alpha_{ji}\equiv \alpha_{ij} \mod p^{\min\{\lambda_i,\lambda_j\}}$ (respectively, $\alpha_{ji}\equiv -\alpha_{ij} \mod p^{\min\{\lambda_i,\lambda_j\}}$).
\end{cor}
Let $\nabla$ be a symplectic self-duality of a finite abelian $p$\dash group $L$ with a maximal isotropic subgroup $G$.
Since $G$ is a finite $p$\dash group, there exists a unique non-decreasing sequence $\lambda=(\lambda_1,\ldots,\lambda_l)$ such that $G$ is isomorphic to $\ctorsion\lambda$.
Fixing an isomorphism $\phi:G\tilde\to\ctorsion\lambda$ gives an extension
\begin{equation}
  \xses{\xi_\phi}{\ctorsion\lambda}L{\dtorsion\lambda}fq.
\end{equation}
The above extension is equivalent to $\xi(\alpha_\phi)$ for some matrix $\alpha_\phi$ as explained above.
In what follows, an automorphism of $\ctorsion\lambda$ will be identified with an automorphism of $\Z^l$ which preserves $p^\lambda(\Z^l)$, i.e., an $l\times l$\dash matrix with integer entries and determinant $\pm 1$ such that $p^{-\lambda}\circ \sigma\circ p^\lambda$ is also an integer matrix.
\begin{lemma}
  \label{lemma:transformation}
  Let $\phi$ and $\phi'$ be two isomorphisms $G\to \ctorsion\lambda$.
  Let $\sigma=\phi'\circ \phi\inv$.
  Then $\alpha_{\phi'}=\sigma\alpha_\phi\transpose\sigma$.
\end{lemma}
\begin{proof}
  Since $\xi_\phi=\xi(\alpha_\phi)$, there is a commutative diagram
  \begin{equation}
    \label{eq:35}
    \xymatrix{
      & 0 \ar[r] & \Z^l\ar[r]^{p^\lambda} \ar[d]_{\alpha_\phi} \ar@{}[dr]|\square & \Z^l \ar[r]^{p^{-\lambda}} \ar[d]^\beta & \dtorsion\lambda \ar@{=}[d] \ar[r] & 0\\
      (\xi_\phi) & 0 \ar[r] & \ctorsion\lambda \ar[d]_\sigma \ar[r]^j & L \ar@{=}[d] \ar[r]^q & \dtorsion\lambda \ar[r] & 0\\
      (\xi_{\phi'}) & 0 \ar[r] & \ctorsion\lambda \ar[r]_{j\circ \bar \sigma \inv} & L \ar[r]_{\hat \sigma \inv \circ q} & \dtorsion\lambda \ar[r]\ar[u]_{\hat \sigma}& 0
    }
  \end{equation}
  whence
  \begin{equation*}
    \xymatrix{
      & 0 \ar[r] & \Z^l \ar[r]^{p^\lambda} \ar[d]_{\bar \sigma\circ \alpha_\phi \circ \transpose{\sigma}} \ar@{}[dr]|\square & \Z^l \ar[r]^{p^{-\lambda}} \ar[d]^{\beta\circ \transpose{(\sigma^\lambda)}} & \T[p^\lambda] \ar@{=}[d] \ar[r] & 0\\
      (\xi_{\phi'}) &0 \ar[r] & \Z/p^\lambda \ar[r]_{j\circ \bar \sigma \inv} & L \ar[r]_{\hat {\bar \sigma}\inv \circ q} & \T[p^\lambda] \ar[r] & 0
    }
  \end{equation*}
  from which one may conclude that $\xi_{\phi'}=\xi(\sigma\alpha\transpose\sigma)$.
\end{proof}
\begin{defn*}
  [Bipartite matrix]
  \label{defn:bipartite-matrix}
  An $l\times l$ matrix $\alpha$ is said to be bipartite if there exists a partition $\{1,\ldots,l\}=S\coprod T$, where $S$ and $T$ are non-empty, such that $\alpha_{ij}\equiv 0\mod p^{\min\{\lambda_i,\lambda_j\}}$ if $i$ and $j$ are both in $S$ or both in $T$.
\end{defn*}
\begin{theorem}
  \label{theorem:finite-standard-triple}
  Let $\nabla$ be a symplectic self-duality on a finite abelian $p$\dash group $L$.
  Let $G$ be a maximal isotropic subgroup of $L$.
  Consider the following statements
  \begin{enumerate}
  \item \label{item:5}The triple $(L,\nabla,G)$ is  standard.
  \item \label{item:6}There exists an isomorphism $\phi:G\to \ctorsion\lambda$ such that $\xi_\phi=\xi(\alpha)$ for a bipartite matrix $\alpha$.
  \end{enumerate}
  Then \ref{item:5} implies \ref{item:6}.
  If $p\neq 2$ then \ref{item:5} and \ref{item:6} are equivalent.
\end{theorem}
\begin{proof}
  In the finite case, the condition \ref{item:4} is equivalent to $\alpha$ being bipartite for some isomorphism $\phi:G\to \ctorsion\lambda$.
  Therefore, Theorem~\ref{theorem:finite-standard-triple} follows from Theorem~\ref{theorem:standard-triples}.
\end{proof}
\begin{theorem}
  [Example of a non-standard triple]
  \label{theorem:non-standard-triple}
  Let $p$ be an odd prime, and let $N\geq s\geq 4$ be two integers.
  Let $\lambda=(p^{sN},p^{(s-1)N},\ldots,p^N)$.
  Let
  \begin{equation*}
    \alpha=
    \left(
      \begin{array}{cccccc}
        0 & p^{s-2} & p^{s-3} & \cdots & p & 1\\
        -p^{s-2} & 0 & p^{s-4} & \cdots & 1 & 0\\
        -p^{s-3} & -p^{s-4} & 0 & \cdots & 0 & 0\\
        \vdots & \vdots & \vdots & \ddots & \vdots & \vdots\\
        -p & -1 & 0 & \cdots & 0 & 0\\
        -1 & 0 & 0 & \cdots & 0 & 0
      \end{array}
     \right).
  \end{equation*}
  Let $\xi(\alpha)$ be the extension over $\dtorsion\lambda$ with kernel $\ctorsion\lambda$ as in (\ref{eq:22}).
  Suppose that $\nabla$ is a symplectic self-duality of $L$ for which $\ctorsion\lambda$ is a maximal isotropic subgroup ($\nabla$ exists by Theorem~\ref{theorem:anti-dual-extensions}).
  Then $(L, \nabla, \ctorsion\lambda)$ is not a standard triple.
\end{theorem}
\begin{proof}
  By Theorem~\ref{theorem:finite-standard-triple} and Lemma~\ref{lemma:transformation} it suffices to show that there does not exist any automorphism $\sigma$ of $\ctorsion\lambda$ such that $\sigma\alpha\transpose\sigma$ is bipartite.
  The transformations $\alpha\mapsto \sigma\alpha\transpose\sigma$ can all be achieved by a sequence of row and column operations of the following types:
  \begin{enumerate}
  \item [$\beta_{ia}$: ] multiplication of the $i$th row and $i$th column by an integer $a$ such that $(a,p)=1$,
  \item [$\pi_{ij}$: ] interchange of $i$th and $j$th rows and columns when $\lambda_i=\lambda_j$,
  \item [$\theta_{ija}$: ] addition of $a$ times the $j$th column to the $i$th column, followed by addition of $a$ times the $j$th row to the $i$th row when $p^{\max\{0,\lambda_i-\lambda_j\}}$ divides $a$.
  \end{enumerate}
  None of these operations can change the valuation of any entry of $\alpha$ that lies on or above the skew-diagonal.
  In particular, $\alpha$ is sparser than any matrix that can be obtained from it by transformations of the above type.
  Therefore, it suffices to show that $\alpha$ itself is not bipartite.
  For any partition of $\{1,\ldots,s\}=S\coprod T$ as in Definition~\ref{defn:bipartite-matrix},
  one may assume, without loss of generality that $1\in S$.
  Since all but the first entry of the first row are non-zero, $\{2,\ldots,s\}\subset T$.
  But both $2$ and $s$ can not be in $T$, since the entry in the second row and  $s$th column is $1$.
  Therefore no such partition can exist, so $\alpha$ is not bipartite.
\end{proof}
\begin{theorem}
  \label{theorem:homogeneous}
  Let $\nabla$ be a symplectic self-duality on a finite abelian group $L$.
  If $G$ is a maximal isotropic subgroup isomorphic to $(\Z/p^k\Z)^l$ for some positive integers $k$ and $l$ and some odd prime $p$, then $(L,\nabla,G)$ is a standard triple.
\end{theorem}
\begin{proof}
  Fix an isomorphism of $G$ with $(\Z/p^k\Z)^l$.  Then
  $L$ corresponds to the extension $\xi(\alpha)$, where $\alpha_{ij}+ \alpha_{ji}\equiv 0\mod p^k$  for all $i$, $j$.
  Using the row and column operations in the proof of Theorem~\ref{theorem:non-standard-triple} we will reduce $\alpha$ to a bipartite matrix.
  We may assume that $l\geq 2$.
  Find an entry of $\alpha$ with lowest valuation at $p$.
  By interchanging rows and columns, bring this entry to the $(1,2)$th position.
  Use it to clear out all the other entries in the first row and in the second column (and therefore, by skew-symmetry, all the entries in the second row and in the first column) by operations of the type $\theta_{ija}$.  These operations are admissible since $\lambda_i=\lambda_j$.

  Repeat the above process with the $(l-2)\times (l-2)$ submatrix in the bottom right corner (this does not disturb the already cleared rows and columns) to eventually obtain a bipartite matrix when $S$ and $T$ are taken to be the even and odd numbers in $\{1,\ldots,l\}$ respectively.
\end{proof}
\begin{theorem}
  \label{theorem:Z_p^l}
  Let $\nabla$ be a symplectic self-duality on a \lca{} $L$.
  If $G$ is a compact open maximal isotropic subgroup isomorphic to $\Z_p^l$ for some odd prime $p$ and some positive integer $l$ then $(L,\nabla,G)$ is a standard triple.
\end{theorem}
\begin{proof}
 Choose an isomorphism $\phi:G\rightarrow \Z_p^l$.  We claim that a $\xi\in E(\hat{G},G)$ can be uniquely represented by an $l\times l$ matrix $\alpha\in M_l(\Z_p)$ using $\phi$.  Furthermore, a different choice of $\phi$ results in a different matrix, namely $\sigma \alpha\transpose\sigma$, for some $\sigma\in GL_l(\Z_p)$.

 Indeed, $E(\hat{G},G)=E(\varinjlim\widehat{G_k},G)$, where $G_k=G/p^kG$, so that $E(\hat{G},G)=\varprojlim E(\widehat{G_k},G)$.
 Using a projective resolution of $(\Z/p^k\Z)^l$ as at the beginning of this section, $\phi$ can be used to identify $E(\widehat{G_k},G)$ with $M_l(\Z/p^k\Z)$ and the claim follows.  Now the rest of the proof of Theorem~\ref{theorem:homogeneous} goes through verbatim.
  \end{proof}
\section{Topological torsion groups, divisible groups, non-standard pairs}
\label{sec:topol-tors-groups}
We recall some basic facts about the primary decomposition of a topological torsion group, referring the reader to \cite[Chapters~2 and~3]{0509.22003} for further details.
For each prime $p$, a topological torsion group $L$ has a unique closed subgroup $L_p$ which is a topological $p$\dash group such that $L$ is the restricted product $\rdp_p(L_p,G_p)$ for some compact open subgroups $G_p$ of $L_p$.  Recall that the restricted product $\rdp_p(L_p,G_p)$ is defined to be the subgroup of $\prod_p L_p$ that consists of sequences $\{x_p\}$ with all but finitely many $x_p$ in $G_p$.  Thus
it is enough to specify $G_p$ for all but finitely many primes $p$.
If $L'$ is another topological torsion group with the primary decomposition $\rdp_p(L_p',G_p')$ then any map between $L$ and $L'$ respects the decompositions, in particular any isomorphism $\phi:L\to L'$ maps $L_p$ isomorphically onto $L'_p$ for all $p$ in such a way that $\phi(G_p)=G'_p$ for all but finitely many $p$.
Conversely, given isomorphisms $\phi_p:L_p\to L'_p$ such that $\phi(G_p)=G'_p$ for all but finitely many $p$, there exists a unique isomorphism $\phi:L\to L'$ whose restriction to $L_p$ is $\phi_p$.
\begin{theorem}
  \label{theorem:toptor}
  Let $\nabla$ be a symplectic self-duality of a topological torsion group $L$.
  Let $L=\rdp_p(L_p,G_p)$ be the primary decomposition of $L$.
  Then
  \begin{enumerate}
  \item \label{item:8} The restriction $\nabla_p$ of $\nabla$ to $L_p$ takes values in $\widehat{L_p}$ and is a symplectic self-duality of $L_p$.
  \item \label{item:9} The compact open subgroup $G_p$ is a maximal isotropic subgroup of $L_p$ for all but finitely many $p$.
  \item \label{item:7} The pair $(L,\nabla)$ is  standard  if and only if $(L_p,\nabla_p)$ is a standard pair for each $p$ and $(L_p,\nabla_p,G_p)$ is a standard triple for all but finitely many $p$.
  \end{enumerate}
\end{theorem}
\begin{proof}
  Since the primary decomposition of $\hat{L}$ is $\rdp_p(\widehat{L_p},G_p^\perp)$ we see that \ref{item:8} follows.
  Furthermore, for all but finitely many $p$ the map $\nabla_p: G_p\rightarrow G_p^\perp$ is an isomorphism, so that \ref{item:9} follows.

  It remains to prove \ref{item:7}.
  Suppose that $\phi:A\times \hat A\to L$ is an isomorphism that takes the standard symplectic self-duality to $\nabla$.
  Note that the induced symplectic self-duality on $A_p\times \widehat{A_p}$ is also standard.
  Let $\rdp(A_p,B_p)$ be the primary decomposition of $A$.
  Then $\rdp(A_p\times\widehat{A_p},B_p\times B_p^\perp)$ is the primary decomposition of $A\times \hat A$.
  Therefore the isomorphisms $\phi_p:A_p\times \widehat A_p\to L_p$ ensure that $(L_p,\nabla_p)$ is a standard pair for all $p$ and that $(L_p,\nabla_p,G_p)$ a standard triple for all but finitely many $p$.

  Conversely, suppose that $(L_p,\nabla_p)$ is a standard pair for all $p$ and  $(L_p,\nabla_p,G_p)$ is a standard triple for all but finitely many $p$.
  Then there exist topological $p$\dash groups $A_p$ and isomorphisms $\phi_p:A_p\times \widehat{A_p}$ which take the standard symplectic self-duality on $A_p\times \widehat{A_p}$ to $\nabla_p$ for all primes $p$.
  Furthermore, for all but finitely many $p$, there exist compact open subgroups $B_p$ of $A_p$ such that that $\phi_p(B_p\times B_p^\perp)=G_p$.
  Set $A=\rdp(A_p,B_p)$.
  Then $\hat A=\rdp(\widehat{A_p},B_p^\perp)$.
  The isomorphism $\phi: A\times\hat{A}\rightarrow L$ whose restriction to $A_p\times \widehat{A_p}$ is $\phi_p$ takes the standard symplectic form on $A\times \hat{A}$ to $\nabla$, thereby demonstrating the standardness of $(L,\nabla)$.
\end{proof}
\begin{theorem}
  [Non-standard pairs]
  \label{theorem:non-standard-pair}
  For each odd prime $p$, let $(L_p,\nabla_p,G_p)$ be a triple from Theorem~\ref{theorem:non-standard-triple}.
  Let $L=\rdp(L_p,G_p)$ be a restricted direct product over all odd primes.
  Then $\nabla$ is a symplectic self-duality of $L$ such that $(L,\nabla)$ is not a standard pair.
  In fact, $L$ is not isomorphic to $A\times \hat A$ for any \lca{} $A$.
  \end{theorem}
\begin{proof}
  That $(L,\nabla)$ is not standard is a consequence of Theorems~\ref{theorem:non-standard-triple} and~\ref{theorem:toptor}.
  Suppose that $\phi:A\times \hat A\to L$ is an isomorphism for some \lca{} $A=\rdp(A_p,B_p)$.
  Then $\phi$ maps $B_p\times B_p^\perp$ onto $G_p$ for all but finitely many $p$.
  This forces $G_p$ to satisfy \ref{item:4} for all but finitely many $p$, which contradicts the proof of Theorem~\ref{theorem:non-standard-triple}.
\end{proof}
\begin{theorem}
  \label{theorem:divisible}
  Let $\nabla$ be a symplectic self-duality of a divisible \lca{} $L$.
  Then $(L,\nabla)$ is a standard pair.
\end{theorem}
\begin{proof}
  By the classification theorem for divisible self-dual \lca{s} \cite[Corollary~6.14]{MR1173767},
  \begin{equation*}
    L\cong \R^n\times \Q^{\oplus\lambda}\times \hat \Q^\lambda\times \rdp_{p\text{ prime}} (L_p,G_p)
  \end{equation*}
  where $L_p=\Q_p^{n_p}$ and $G_p=\Z_p^{n_p}$ for some cardinal number $\lambda$ and integers and integers $n$ and $n_p$.
  Theorem~\ref{theorem:reduction-to-residual} can be used to reduce to the case where $L$ is residual, which means that $L=\rdp_p(L_p,G_p)$.
  Let $\nabla_p$ be the restriction of $\nabla$ to $L_p$.  The morphism
  $\nabla_p$ takes values in $\widehat{L_p}$, giving a symplectic self-duality of $L_p$.
  By Corollary~\ref{cor:fdvs}, $(L_p,\nabla_p)$ is a standard pair for all $p$.
  For all but finitely many $p$, $G_p$ is a maximal isotropic subgroup of $L_p$.
  For these $p$, $(L_p,\nabla_p,G_p)$ is a standard triple by Theorem~\ref{theorem:Z_p^l}.
  Thus Theorem~\ref{theorem:divisible} follows from Theorem~\ref{theorem:toptor}.
\end{proof}
The  treatment below is based on the idea of the canonical filtration of a \lca{} that can be found in \cite{MR2329311}.  Namely, to each \lca{} $L$ one associates a unique filtration by closed subgroups $$L_{\T}\subset F_{\Z}L\subset L$$ where $L_{\T}$ is compact and connected, $F_{\Z}L/L_{\T}\cong \R^n\oplus L_\toptor$ where $L_\toptor$ is a topological torsion group, and $L_\Z:=L/F_{\Z}L$ is discrete and torsion-free. More precisely, using the Pontryagin-van Kampen structure theorem mentioned in Section~\ref{sec:survey} we may remove $\R^n$ and assume that $L$ has a compact open subgroup $G$.
Let $L_\T$ denote the identity component of $G$.
Let $F_\Z L$ denote the pre-image in $L$ of the torsion subgroup of $L/L_\T$.
Neither of these subgroups depends on the choice of $G$.  Note that  the dual filtration on $\hat{L}$, i.e., $$\widehat{L_\Z}=\widehat{L/F_\Z L}\subset\widehat{L/L_\T}\subset \hat{L}$$ is the canonical filtration on $\hat{L}$.

\begin{remark}\label{HSremark}
  Let $\nabla$ be a symplectic self-duality of a \lca{} $L$.  One immediately notes the following facts.  The isomorphism $\nabla$ identifies $L_\T$ with $\widehat{L_\Z}$, $F_\Z L$ with $L_\T^\perp$, and $L_\Z$ with $\widehat{L_\T}$. The subgroup $L_\T$ is isotropic, since any map from a connected group to a discrete group (and $\widehat{L_\T}$ is discrete) is trivial.  Recall that for an isotropic subgroup $G\subset L$, we set $L_G$ to be the subquotient $\nabla^{-1}G^\perp/G$, which inherits a symplectic self-duality $\nabla_G$.  In our case we see that $L_{L_\T}$ is exactly $F_\Z L/L_\T$, i.e., $$L_{L_\T}=\R^n\oplus L_\toptor.$$
\end{remark}
 We also need the following \cite[Proposition~3.3]{MR2329311}.
\begin{theorem}\label{HSlemma}
If $\iota:A'\rightarrow A$ is an open embedding of \lca{}s and $D$ is a divisible \lca{}, then $\iota^*:\text{Hom}(A,D)\rightarrow \text{Hom}(A',D)$ is surjective.
\end{theorem}
We now describe a situation where the question of standardness can be reduced to topological torsion groups.
\begin{theorem}
  \label{theorem:reduction-to-toptor}
  Let $\nabla$ be a symplectic self-duality of a \lca{} $L$.
  If $L_\T$ splits from $L$ then $(L,\nabla)$ is a standard pair if $(L_{L_\T},\nabla_{L_\T})$ is.
\end{theorem}
\begin{proof}
  As pointed out in Remark~\ref{HSremark}, the subgroup $L_\T$ is isotropic, and using the peeling lemma, $(L,\nabla)$ is the orthogonal sum of $(L_\T\times \widehat{L_\T},\nabla')$ and $(L_{L_\T},\nabla_{L_\T})$ where $\nabla'$ is a symplectic self-duality of $L_\T\times \widehat{L_\T}$ for which $L_\T$ is isotropic.

By Lemmas~\ref{lemma:polarization-implies-standard} and~\ref{lemma:split-maximal-isotropics}, to complete the proof, it suffices to show that  the component $\nabla':\widehat{L_\T}\rightarrow L_\T$ is divisible by $2$.  We claim that in fact any map from $\widehat{L_\T}$ to $L_\T$ is divisible by $2$.  Note that since $\widehat{L_\T}$ is discrete and torsion-free, the multiplication by $2$ map is an open embedding.  Of course $L_\T$ itself is divisible and so by Theorem~\ref{HSlemma} we are done.
\end{proof}
It is worth pointing out that there exist pairs $(L,\nabla)$ where $L_\T$ does not split off, and we have seen examples (Theorem~\ref{theorem:non-standard-pair}) where $(L_{L_\T},\nabla_{L_\T})$ is non-standard.  Thus both hypotheses of Theorem~\ref{theorem:reduction-to-toptor} are non-trivial, however, below we describe some cases where they are satisfied.
\begin{cor}
  \label{cor:torus-toptor}
  Let $\nabla$ be a symplectic self-duality of a \lca{} $L$.
  If $L_\T$ is isomorphic to $\T^n$ for some positive integer $n$, then $(L,\nabla)$ is a standard pair if $(L_{L_\T},\nabla_{L_\T})$ is.
\end{cor}
\begin{proof}
 The subgroup $L_\T$ splits from $L$ \cite[6.16]{0509.22003} so Theorem~\ref{theorem:reduction-to-toptor} applies.
\end{proof}
\begin{cor}
  \label{cor:finite-toptor}
  Let $\nabla$ be a symplectic self-duality of a \lca{} $L$.
  If $L_\toptor$ is finite then $(L,\nabla)$ is a standard pair.
\end{cor}
\begin{proof}
  Using the peeling lemma as explained in the proof of Theorem~\ref{theorem:reduction-to-residual} we may assume that $L$ has a compact open subgroup, i.e., that $L$ has no $\R^n$ part.
  Since $L_\toptor$ is finite, $F_\Z L$ is compact.  Consider the exact sequence \begin{equation}\label{ses}0\rightarrow L_\T\rightarrow L\rightarrow L/L_\T\rightarrow 0;\end{equation} we claim that it splits.  Indeed  $\nabla$ identifies $F_\Z L$ with $\widehat{L/L_\T}$ so that $L/L_\T$ is discrete, i.e., $L_\T\rightarrow L$ is an open embedding and so by Theorem~\ref{HSlemma} it splits off since $L_\T$ is divisible.  Furthermore, note that $(L_{L_\T},\nabla_{L_\T})$ is standard by Corollary~\ref{cor:finite} as $L_{L_\T}$ is a finite group.
\end{proof}


\begin{thebibliography}{10}

\bibitem{0509.22003}
D.~L. Armacost.
\newblock {\em {The structure of locally compact abelian groups.}}
\newblock {Monographs and Textbooks in Pure and Applied Mathematics, 68. New
  York-Basel; Marcel Dekker, Inc. VII, 154 p.}, 1981.

\bibitem{MR0578649}
D.~L. Armacost and W.~L. Armacost.
\newblock Uniqueness in structure theorems for {LCA} groups.
\newblock {\em Canad. J. Math.}, 30(3):593--599, 1978.

\bibitem{homologicalalgebra}
H.~Cartan and S.~Eilenberg.
\newblock {\em Homological Algebra}.
\newblock Princeton University Press, 1956.

\bibitem{MR0269775}
L.~Corwin.
\newblock Some remarks on self-dual locally compact {A}belian groups.
\newblock {\em Trans. Amer. Math. Soc.}, 148:613--622, 1970.

\bibitem{MR1475126}
L.~Fuchs and K.~H. Hofmann.
\newblock Extensions of compact abelian groups by discrete ones and their
  duality theory. {I}.
\newblock {\em J. Algebra}, 196(2):578--594, 1997.

\bibitem{MR1651168}
L.~Fuchs and K.~H. Hofmann.
\newblock Extensions of compact abelian groups by discrete ones and their
  duality theory. {II}.
\newblock In {\em Abelian groups, module theory, and topology (Padua, 1997)},
  volume 201 of {\em Lecture Notes in Pure and Appl. Math.}, pages 205--225.
  Dekker, New York, 1998.

\bibitem{MR2329311}
N.~Hoffmann and M.~Spitzweck.
\newblock Homological algebra with locally compact abelian groups.
\newblock {\em Adv. Math.}, 212(2):504--524, 2007.

\bibitem{MR578375}
R.~Howe.
\newblock On the role of the {H}eisenberg group in harmonic analysis.
\newblock {\em Bull. Amer. Math. Soc. (N.S.)}, 3(2):821--843, 1980.

\bibitem{Mackey49}
G.~W. Mackey.
\newblock A theorem of {S}tone and von~{N}eumann.
\newblock {\em Duke Math. J.}, 16:313--326, 1949.

\bibitem{MR1116553}
D.~Mumford.
\newblock {\em Tata lectures on theta. {III}}, volume~97 of {\em Progress in
  Mathematics}.
\newblock Birkh\"auser Boston Inc., Boston, MA, 1991.
\newblock With the collaboration of Madhav Nori and Peter Norman.

\bibitem{iafhg}
A.~Prasad and M.~K. Vemuri.
\newblock Inductive algebras for finite {H}eisenberg groups.
\newblock To appear in \emph{Comm. Alg.}
  (\href{http://arxiv.org/abs/0806.4064}{arXiv:0803.2921}).

\bibitem{MR0247375}
M.~Rajagopalan and T.~Soundararajan.
\newblock Structure of self-dual torsion-free metric {${\rm LCA}$} groups.
\newblock {\em Fund. Math.}, 65:309--316, 1969.

\bibitem{MR1280069}
R.~Ranga~Rao.
\newblock Symplectic structures on locally compact abelian groups and
  polarizations.
\newblock {\em Proc. Indian Acad. Sci. Math. Sci.}, 104(1):217--223, 1994.
\newblock K. G. Ramanathan memorial issue.

\bibitem{MR1503234}
E.~R. van Kampen.
\newblock Locally bicompact abelian groups and their character groups.
\newblock {\em Ann. of Math. (2)}, 36(2):448--463, 1935.

\bibitem{pre05272580}
M.~K.~Vemuri.
\newblock {Realizations of the canonical representation.}
\newblock {\em Proc. Indian Acad. Sci., Math. Sci.}, 118(1):115--131, 2008.

\bibitem{MR0035776}
N.~Y. Vilenkin.
\newblock Direct decompositions of topological groups. {I}, {II}.
\newblock {\em Amer. Math. Soc. Translation}, 1950(23):109, 1950.

\bibitem{MR0165033}
A.~Weil.
\newblock Sur certains groupes d'op\'erateurs unitaires.
\newblock {\em Acta Math.}, 111:143--211, 1964.

\bibitem{Weyl50}
H.~Weyl.
\newblock {\em The theory of groups and quantum mechanics}.
\newblock Dover Publications, Inc., 1950.
\newblock Translated from the second (revised) {G}erman edition by
  {H}.~{P}.~{R}obertson.

\bibitem{MR1173767}
S.~L. Wu.
\newblock Classification of self-dual torsion-free {LCA} groups.
\newblock {\em Fund. Math.}, 140(3):255--278, 1992.

\end{thebibliography}
\end{document}